\documentclass{amsart}
\usepackage[utf8]{inputenc}
\usepackage{amsmath}
\usepackage{amssymb}
\usepackage{caption}
\usepackage{amsthm}
\usepackage[hidelinks]{hyperref}
\usepackage{cleveref}
\usepackage[justification=centering]{caption}
\crefname{lemma}{Lemma}{Lemmas}
\crefname{theorem}{Theorem}{Theorems}
\usepackage{xcolor}
\usepackage{comment}
\usepackage{graphicx}
\usepackage{tikz}
\usepackage{tikz-cd}
\usepackage{mathrsfs}
\usepackage{mathtools}
\usepackage{enumitem}
\usetikzlibrary{shapes.geometric}
\makeatletter
\def\@settitle{\begin{center}%
  \baselineskip14\p@\relax
  \bfseries
  \uppercasenonmath\@title
  \@title
  \ifx\@subtitle\@empty\else
     \\[1ex]\uppercasenonmath\@subtitle
     \footnotesize\mdseries\@subtitle
  \fi
  \end{center}%
}
\def\subtitle#1{\gdef\@subtitle{#1}}
\def\@subtitle{}
\makeatother

\newtheorem{theorem}{Theorem}[section]

\newtheorem{proposition}{Proposition}[theorem]
\newtheorem{lemma}[theorem]{Lemma}

\theoremstyle{remark}
\newtheorem{remark}[theorem]{Remark}

\theoremstyle{remark}

\usepackage{chngcntr}
\usepackage{graphicx} 
\usepackage{float}
\counterwithout{equation}{section}
\counterwithout{theorem}{section}

\begin{document}

\title{Anosov groups that are indiscrete in rank one}

\begin{abstract}
We exhibit Anosov subgroups of $\mathsf{SL}_d(\mathbb{R})$ that do not embed discretely in any rank-$1$ simple Lie group of noncompact type, or indeed, in any finite product of such Lie groups. These subgroups are isomorphic to free products $\Gamma * \Delta$, where $\Gamma$ is a uniform lattice in $\mathsf{F}_4^{(-20)}$ and $\Delta$ is a uniform lattice in $\mathsf{Sp}(m,1)$, $m \geq 51$.
\end{abstract}

\author{Sami Douba}
\address{Institut des Hautes \'Etudes Scientifiques, 
Universit\'e Paris-Saclay, 35 route de Chartres, 91440 Bures-sur-Yvette, France}
\email{douba@ihes.fr}

\author{Konstantinos Tsouvalas}
\address{Max Planck Institute for Mathematics in the Sciences, Inselstrasse 22, 04103 Leipzig, Germany}
\email{konstantinos.tsouvalas@mis.mpg.de}

\maketitle

\section{Introduction}
%The theory of Anosov representations of Gromov hyperbolic groups into non compact semisimple Lie groups generalizes that of convex cocompact groups in simple Lie groups of real rank one. The work of  Danciger--Gu\'eritaud--Kassel \cite{DGK} and independently of Zimmer \cite{Zimmer} shows that every projective Anosov representation $\rho:\Gamma \rightarrow \mathsf{SL}_{d}(\mathbb{R})$ can be associated to  a convex cocompact action on some properly convex domain in a real projective space\footnote{In fact when $\rho(\Gamma)$ preserves a properly convex domain in $\mathbb{P}(\mathbb{R}^d)$ then it acts convex cocompactly on some (possibly different) properly convex domain in $\mathbb{P}(\mathbb{R}^d)$.}. However, the authors of this note were not aware of an example in the literature of a Gromov hyperbolic group admitting a projective Anosov representation into a special linear group but not a discrete faithful representation into any product of finitely many rank one Lie groups. 

Throughout this note, a {\em rank-$1$ Lie group} is a Lie group isogenous to the isometry group of an irreducible symmetric space of noncompact type and real rank $1$. Such a symmetric space is isometric (up to scaling) to one of $\mathbb{R}{\bf H}^n, \mathbb{C}{\bf H}^n, \mathbb{H}{\bf H}^n$, $n \geq 2$, or $\mathbb{O}{\bf H}^2$; correspondingly, each rank-$1$ Lie group is isogenous to one of $\mathsf{SO}(n,1)$, $\mathsf{SU}(n,1)$, $\mathsf{Sp}(n,1)$, $n\geq 2$, or $\mathsf{F}_4^{(-20)}$; see, for instance, \cite[Thm.~8.12.2]{Wolf}.

Since their introduction in Labourie's seminal paper on the Hitchin component \cite{Labourie} and the further development of their theory by Guichard--Wienhard \cite{GW12}, Kapovich--Leeb--Porti \cite{KLP17, KLP18a, KLP18b},  Gu\'eritaud--Guichard--Kassel--Wienhard \cite{GGKW}, Bochi--Potrie--Sambarino \cite{BPS}, and others, Anosov representations have emerged as successful higher-rank generalizations of convex cocompact representations into rank-$1$ Lie groups. For a survey on Anosov representations and their strong dynamical, geometric, and topological properties, see \cite{KasselICM}. 

That being said, the authors were not aware of an example in the literature of a Gromov-hyperbolic group that admits an Anosov embedding into a higher-rank Lie group but does not already embed as a convex cocompact subgroup of a rank-$1$ Lie group. This question was raised by Canary \cite[Prob.~50.3]{canary2021anosov}; we thank the referee for this reference. 

The purpose of this note is to furnish such examples. Indeed, we observe the following (proved in \S \ref{proofs}).

\begin{theorem}\label{mainthm0} Let $\Gamma_1$ and $\Gamma_2$ be uniform lattices in $\mathsf{F}_4^{(-20)}$. Then the free product $\Gamma_1 * \Gamma_2$ admits no discrete and faithful representation into any rank-$1$ Lie group.\end{theorem}

One can replace $\Gamma_2$ in the statement of Theorem \ref{mainthm0} with any nontrivial group (indeed, if~$\Gamma_2$ is nontrivial then $\Gamma_1\ast \Gamma_1$ embeds as a subgroup of $\Gamma_1\ast \Gamma_2$). If we replace $\Gamma_2$ with a uniform quaternionic hyperbolic lattice of large dimension, we obtain a stronger conclusion.

\begin{theorem}\label{mainthm1} Let $\Gamma$ be a uniform lattice in $\mathsf{F}_4^{(-20)}$ and $\Delta$ a uniform lattice in $\mathsf{Sp}(m,1)$, where $m\geq 51$. Then the free product $\Gamma \ast \Delta$ admits no discrete and faithful representation into any Lie group isogenous to a product of rank-$1$ Lie groups.\end{theorem}

The proofs of Theorems \ref{mainthm0} and \ref{mainthm1} make crucial use of the rank-$1$ superrigidity theorems of Corlette \cite{Corlette} and Gromov--Schoen \cite{GS}.

On the other hand, using a recent combination theorem of Dey--Kapovich--Leeb \cite{DKL}, we show that the free products in Theorems \ref{mainthm0} and \ref{mainthm1} admit Anosov embeddings. In fact, we establish in \S \ref{freeproducts} the following more general statement.

\begin{theorem}\label{anosovfreeproducts}
The property of admitting an Anosov embedding into a semisimple linear algebraic $\mathbb{R}$-group is closed under finite free products. 
\end{theorem}

Theorem \ref{anosovfreeproducts} also follows from a result announced by Danciger--Gu\'eritaud--Kassel \cite[Prop.~12.5]{DGK} and proved in their forthcoming work \cite{DGK2}.

In particular, the free products in Theorems \ref{mainthm0} and \ref{mainthm1} are examples of linear Gromov-hyperbolic groups that do not admit discrete and faithful representations into any rank-$1$ Lie group. The first examples of such hyperbolic groups were exhibited in \cite[Thm. 1.2 \& 1.7]{TT} as amalgamated products of two copies of a torsion-free uniform  lattice $\Delta<\mathsf{Sp}(m,1)$, $m\geq 2$, along a maximal cyclic subgroup of $\Delta$. It was suggested to the second author by Beatrice Pozzetti that such amalgams over $\mathbb{Z}$ also admit Anosov embeddings, though we do not pursue this here.

\medskip

\noindent {\bf Acknowledgements.} We thank Richard Canary, Fanny Kassel, Beatrice Pozzetti, and Florian Stecker for helpful discussions, as well as the referee for a careful reading of this note and for useful comments and suggestions. We also thank the IH\'ES for providing excellent working conditions. The first author was supported by the Huawei Young Talents Program. The second author was supported by the European Research Council (ERC) under the European's Union Horizon 2020 Research and Innovation Programme (ERC starting grant DiGGeS, grant agreement No 715982).

\section{Preliminaries}

Let $\mathsf{G}$ be a finite-center real semisimple Lie group with finitely many connected components and $\mathsf{K}$ a maximal compact subgroup of $\mathsf{G}$. Let $\mathfrak{a}$ be a Cartan subspace of $\mathfrak{g} = \mathrm{Lie}(\mathsf{G})$ and $\overline{\mathfrak{a}}^+ \subset \mathfrak{a}$ a dominant Weyl chamber, so that there exists a Cartan decomposition $\mathsf{G} = \mathsf{K} \exp(\overline{\mathfrak{a}}^+) \mathsf{K}$. Let $\mu: \mathsf{G} \rightarrow \overline{\mathfrak{a}}^+$ be the associated Cartan projection. Given a non-empty set $\Theta$ of simple restricted roots of $\mathfrak{g}$, a representation $\rho: \Gamma \rightarrow \mathsf{G}$ of a finitely generated group $\Gamma$ is {\em $\Theta$-Anosov} if there exist $c, C > 0$ such that for every $\gamma \in \Gamma$ and $\theta \in \Theta$,
\begin{equation}\label{Anosovdef}
\theta \big( \mu(\rho(\gamma)) \big) \geq c  \big| \gamma \big|_\Gamma - C,
\end{equation}
where $| \cdot |_\Gamma$ denotes word length with respect to some fixed finite generating set of $\Gamma$. That this characterization is equivalent to Labourie's original dynamical definition was established independently by Kapovich--Leeb--Porti \cite{KLP18b} and Bochi--Potrie--Sambarino \cite{BPS}. It is visible from this definition that if $\rho|_{\Gamma_0}$ is $\Theta$-Anosov for some finite-index subgroup $\Gamma_0<\Gamma$ then $\rho$ is itself $\Theta$-Anosov.

Observe that when $\mathsf{G}$ is a rank-$1$ Lie group then (\ref{Anosovdef}) simply says that $\rho$ is a quasi-isometric embedding, i.e., that $\rho$ is convex cocompact. We say a subgroup $\Gamma < \mathsf{G}$ is {\em $\Theta$-Anosov} if $\Gamma$ is the image of a $\Theta$-Anosov representation into~$\mathsf{G}$.

We now clarify condition (\ref{Anosovdef}) for the specific case $\mathsf{G} = \mathsf{SL}_d(\mathbb{R})$.
Given $g \in \mathsf{SL}_{d}(\mathbb{R})$, denote by $\mu_1(g) \geq \mu_2(g)\geq \ldots \geq \mu_d(g)$ the logarithms of the singular values of $g$ in non-increasing order (counting multiplicity). A representation $\rho:\Gamma \rightarrow \mathsf{SL}_d(\mathbb{R})$ is {\em $P_i$-Anosov}, $1\leq i\leq d-1$, if there exist $c, C>0$ such that $$\mu_i\big(\rho(\gamma)\big)-\mu_{i+1}\big(\rho(\gamma)\big)\geq c|\gamma|_{\Gamma}-C$$ for every $\gamma \in \Gamma$. We will use repeatedly the following fact. 

\begin{lemma} \label{comm} Suppose a finite-index normal subgroup $\Gamma_0<\Gamma$ embeds as a $P_1$-Anosov subgroup of $\mathsf{SL}_d(\mathbb{R})$. Then $\Gamma$ embeds as a $P_1$-Anosov subgroup of $\mathsf{SL}_r(\mathbb{R})$ for some $r\in \mathbb{N}$.
\end{lemma}

\begin{proof} Let $\rho:\Gamma_0 \xhookrightarrow{} \mathsf{SL}_{d}(\mathbb{R})$ be the inclusion and ${\rho^\text{ind}:\Gamma \rightarrow \mathsf{SL}_{dm}^\pm(\mathbb{R})}$ the induced representation (see, for instance, \cite[Section~3.3]{FH91}), where $m=[\Gamma:\Gamma_0]$. Since $\rho$ is faithful, the same is true for $\rho^\text{ind}$. Set $\ell = dm+1$ and let~$\hat{\rho} : \Gamma \rightarrow \mathsf{SL}_\ell(\mathbb{R})$ be the composition of $\rho^\text{ind}$ with the embedding $\mathsf{SL}_{dm}^\pm(\mathbb{R}) \rightarrow \mathsf{SL}_\ell(\mathbb{R})$ given by 
\begin{equation*} A \mapsto \begin{pmatrix}A & 0 \\ 0 &  \det(A) \end{pmatrix}. \end{equation*}
Since the restriction $\rho^\text{ind} |_{\Gamma_0}$ is an $m$-fold direct sum of $P_1$-Anosov representations, and since~$\hat{\rho}|_{\Gamma_0}$ is obtained from $\rho^\text{ind} |_{\Gamma_0}$ by inserting a $1$ on the diagonal, we have that $\hat{\rho}|_{\Gamma_0}$ is $P_m$-Anosov in~$\mathsf{SL}_{\ell}(\mathbb{R})$, and so the latter is also true for $\hat{\rho}$. If $\ell$ is even, we replace $\hat{\rho}$ with the representation obtained from $\hat{\rho}$ by inserting a $1$ on the diagonal (we still call the latter representation~$\hat{\rho}$), and increase $\ell$ by $1$. If $\ell$ is odd, we keep $\hat{\rho}$ and $\ell$ as is. Note that in any case $\hat{\rho}$ remains $P_m$-Anosov. 

Now consider the $m^{\textup{th}}$ exterior power $\bigwedge^m \hat{\rho}:\Gamma \rightarrow \mathsf{SL}\left(\bigwedge^{m}\mathbb{R}^{\ell}\right)$. Since $\ell$ is odd, we have that~$\bigwedge^{m}\hat{\rho}$ is faithful. Moreover, since $\hat{\rho}$ is $P_m$-Anosov we have that $\bigwedge^{m}\hat{\rho}$ is $P_1$-Anosov.\end{proof}

%Let $\mathbb{H}=\mathbb{R}\oplus \mathbb{R}i\oplus \mathbb{R}j\oplus \mathbb{R}k$ be Hamilton's quaternion algebra. For $z=a+bi+cj+dk \in \mathbb{H}$, $\overline{z}=a-bi-cj-dk$ denotes the conjugate of $z \in \mathbb{H}$. For a matrix $g=(g_{ij})_{ij} \in \textup{M}_{n \times n}(\mathbb{H})$, $g^{\ast}=(\overline{g_{ji}})_{ij}$ is the conjugate transpose of $g$. For $m \geq 1$, let $J_{m}=\textup{diag}\big(1,\ldots,1,-1\big)$. The projective space ${\bf P}(\mathbb{H}^{m+1})$ is by definition the set of equivalence classes of vectors in $\mathbb{H}^{m+1}$, where $u,v \in \mathbb{H}^{m+1}$ are equivalent if $u=vz$ for some $z \in \mathbb{H}\smallsetminus \{0\}$. The {\em quaternionic hyperbolic space} $\mathbb{H}{\bf H}^m$ is the open subset of the projective space $\mathbf{P}(\mathbb{H}^m)$ given by $$\mathbb{H}{\bf H}^m=\big\{[z_0:\ldots:z_{m-1}:z_m] \in {\bf P}(\mathbb{H}^{m+1}):\big|z_0\big|^2+\cdots+\big|z_{m-1}\big|^2<\big|z_m\big|^2\big \}$$ We consider the symplectic unitary group $$\mathsf{Sp}(m,1)=\big\{g \in \mathsf{GL}_{m+1}(\mathbb{H}): g^{\ast}J_mg=J_m \big\}$$ and the compact subgroup $\mathsf{Sp}(m)=\big \{g \in \mathsf{GL}_{m}(\mathbb{H}):gg^{\ast}=I_m\big \}$. The Lie group $\mathsf{Sp}(m,1)$ preserves and acts transitively on $\mathbb{H}{\bf H}^m$. The stabilizer of ${\bf v_0}=[\{0\}^{m}:1]$ in $\mathsf{Sp}(m,1)$ is the group $\mathsf{Sp}(m)\times \mathsf{Sp}(1)$ which is also the, unique up to conjugation, maximal compact subgroup of $\mathsf{Sp}(m,1)$.

Following \cite{WM}, we say that two Lie groups $\mathsf{G}_1$ and $\mathsf{G}_2$ are {\em isogenous} if there exist finite index subgroups $\mathsf{G}_i' < \mathsf{G}_i$ and finite normal subgroups $M_i< \mathsf{G}_i'$ such that $\mathsf{G}_1'/M_1\cong \mathsf{G}_2'/M_2$.

The following proposition is a consequence of Archimedean superrigidity of lattices in~$\mathsf{F}_4^{(-20)}$.

\begin{proposition}\label{bd2} Let $\Gamma$ be a lattice in $\mathsf{F}_4^{(-20)}$, and suppose that $\mathsf{G}$ is a rank-$1$ Lie group that is not isogenous to $\mathsf{F}_4^{(-20)}$. Then every representation $\rho:\Gamma \rightarrow \mathsf{G}$ has bounded image.\end{proposition}

\begin{proof} Suppose that $\rho(\Gamma)$ has noncompact closure in $\mathsf{G}$. Corlette's Archimedean superrigidity theorem \cite{Corlette} provides a continuous $\rho$-equivariant totally geodesic embedding $\mathbb{O}{\bf H}^2\xhookrightarrow{} X_{\mathsf{G}}$, where $X_{\mathsf{G}}$ is the symmetric space associated to $\mathsf{G}$. Since $\mathsf{G}$ is not isogenous to $\mathsf{F}_4^{(-20)}$, we may find a totally geodesic embedding $X_{\mathsf{G}}\xhookrightarrow{} \mathbb{H}{\bf H}^m$ for $m \in \mathbb{N}$ large enough. In particular, we obtain a totally geodesic embedding of the Cayley hyperbolic plane $\mathbb{O}{\bf H}^2$ into $\mathbb{H}{\bf H}^m$, but this is impossible by the classification of totally geodesic subspaces of $\mathbb{H}{\bf H}^m$ \cite[Thm. 2.12]{M}. \end{proof}

%\begin{theorem} Let $\mathsf{G}$ be either $\mathsf{Sp}(m,1)$, $m\geq 2$, or $\mathsf{F}_{4}^{(-20)}$, $\mathsf{\Gamma}<\mathsf{G}$ be a lattice and $\mathbb{F}$ be a subfield of $\mathbb{C}$. Suppose that $\rho: \mathsf{\Gamma} \rightarrow \mathsf{GL}_d(\mathbb{F})$ is a representation with infinite image. Then there exists a faithful representation $\tau:\mathsf{GL}_d(\mathbb{F})\rightarrow \mathsf{GL}_d(\mathbb{C})$ such that $\tau\circ \rho:\mathsf{\Gamma}\rightarrow \mathsf{GL}_d(\mathbb{C})$ has unbounded image.\end{theorem}

\section{Proofs of Theorems \ref{mainthm0} \& \ref{mainthm1}}\label{proofs}
\begin{proof}[Proof of Theorem \ref{mainthm0}] Assume for a contradiction that we have a discrete and faithful representation $\rho:\Gamma_1\ast \Gamma_2\rightarrow \mathsf{G}$, where $\Gamma_1$ and $\Gamma_2$ are uniform lattices in $\mathsf{F}_{4}^{(-20)}$, and $\mathsf{G}$ is a rank-$1$ Lie group. Since $\rho$ is discrete and faithful on the factor $\Gamma_1$, we must have that $\mathsf{G}$ is isogenous to $\mathsf{F}_4^{(-20)}$ by Proposition \ref{bd2}. It follows that $\rho(\Gamma_1)$ is a uniform lattice of $\mathsf{G}$ since the virtual cohomological dimension of $\Gamma_1$ is equal to the dimension of $\mathbb{O}{\bf H}^2$; see, for instance, \cite[Prop.~VIII.8.1]{brown1994cohomology}. Since $\rho(\Gamma_1)$ is a lattice in~$\mathsf{G}$, and $\rho$ is discrete and faithful, we deduce that $\Gamma_1$ is of finite index in $\Gamma_1\ast \Gamma_2$. This is absurd since~$\Gamma_2$ is nontrivial.
\end{proof}

To prove Theorem \ref{mainthm1}, we will make use of the following consequence of Gromov--Schoen superrigidity \cite{GS}. Similar arguments can be found in the proofs of \cite[Thm. 8.1]{Kap} and \cite[Thm. 3.1]{CST}. 

%Crucial for the proof of our main theorem are the superrigidity theorem of Corlette \cite{Corlette} and Gromov--Schoen \cite{GS}.  In fact we will use the following theorem which is a consequence of the superrigidity theorems, see \cite[Thm. 3.1]{CST} and the proof in \cite[Thm. 8.1]{Kap}.

\begin{proposition}\label{infinite} Let $\mathsf{G}$ be either $\mathsf{Sp}(m,1)$, $m\geq 2$, or $\mathsf{F}_{4}^{(-20)}$, and let $\Gamma < \mathsf{G}$ be a lattice. Suppose that $\rho: \Gamma \rightarrow \mathsf{GL}_d(\mathbb{R})$ is a representation with infinite image. Then there is a representation $\rho': \Gamma \rightarrow \mathsf{GL}_d(\mathbb{C})$ with unbounded image. 
\end{proposition} 

\begin{proof}
We may assume that $\rho$ has bounded image, so that $\rho(\Gamma) \subset \mathsf{O}(n)$ up to postconjugation. Since $\Gamma$ has Property (T) (see \cite{BdlHV} and the references therein), up to further postconjugation, we have that $\rho(\Gamma) \subset \mathsf{O}(n,\mathbb{K})$ for some number field $\mathbb{K}\subset \mathbb{R}$ \cite[Prop.~6.6]{Rag}. Moreover, since $\Gamma$ is finitely generated, we in fact have $\rho(\Gamma) \subset \mathsf{O}(n, A)$ for some finitely generated subdomain $A \subset \mathbb{K}$. We may now find embeddings $A \subset \mathbb{K}_i$, where $\mathbb{K}_1, \ldots, \mathbb{K}_r$ are local fields, with $\mathbb{K}_1, \ldots, \mathbb{K}_s$ Archimedean and $\mathbb{K}_{s+1}, \ldots, \mathbb{K}_r$ non-Archimedean, so that the diagonal embedding $A \xhookrightarrow{} \prod_{i=1}^r \mathbb{K}_i$ is discrete (see, for instance, \cite[pg.~77,~Thm.~8]{weil1995basic}). We thus obtain from $\rho$ a discrete representation $\Gamma \rightarrow \prod_{i=1}^r \mathsf{GL}_n(\mathbb{K}_i)$. By the superrigidity result of Gromov--Schoen \cite{GS}, we have that the projection $\Gamma \rightarrow \prod_{i={s+1}}^r \mathsf{GL}_n(\mathbb{K}_i)$ is bounded, so that the projection $\Gamma \rightarrow \prod_{i=1}^s \mathsf{GL}_n(\mathbb{K}_i)$ is discrete. Since the latter representation has infinite image, we conclude that at least one of the projections  $\Gamma \rightarrow \mathsf{GL}_n(\mathbb{K}_i)$, $1 \leq i \leq s$, has unbounded image.
\end{proof}

We deduce the following from Proposition \ref{infinite}.

\begin{theorem}\label{bd1} Let $\Delta$ be a lattice in $\mathsf{Sp}(m,1)$, where $m\geq 51$.  Suppose that $\mathsf{H}$ is a semisimple Lie group isogenous to $\mathsf{F}_4^{(-20)}$. Then every representation $\rho:\Delta \rightarrow \mathsf{H}$ has finite image.\end{theorem}

\begin{proof} Let $\mathsf{H}_0$ be a finite-index subgroup of $\mathsf{H}$, and $F_0$ and $F_1$ finite normal subgroups of $\mathsf{H}_0$ and $\mathsf{F}_4^{(-20)}$, respectively, such that $\mathsf{H}_0/F_0\cong \mathsf{F}_{4}^{(-20)}/F_1$. Denote by $\mathfrak{g}$ the $52$-dimensional real Lie algebra of $\mathsf{F}_4^{(-20)}$. Since $F_1$ is central in $\mathsf{F}_4^{(-20)}$, the adjoint representation $\textup{Ad}:\mathsf{F}_4^{(-20)} \rightarrow \mathsf{GL}(\mathfrak{g})$ induces a well-defined representation \hbox{$\psi: \mathsf{H}_0/F_0\rightarrow \mathsf{GL}(\mathfrak{g})$} with finite kernel.

 We now pass to a finite-index subgroup $\Delta_0$ of $\Delta$ such that $\rho(\Delta_0)$ is contained in $\mathsf{H}_0$, and consider the composition $\phi: = \psi\circ \pi\circ \rho:\Delta_0\rightarrow \mathsf{GL}(\mathfrak{g})$, where $\pi$ is the projection $\mathsf{H}_0\rightarrow \mathsf{H}_0/F_0$. Observe that $\rho$ has finite image if and only if $\phi$ does. 

 Now assume that $\phi$ has infinite image. In this case, Proposition \ref{infinite} provides a representation $\phi':\Delta_0 \rightarrow \mathsf{GL}_{52}(\mathbb{C})$ with unbounded image. In particular, by Corlette's Archimedean superrigidity theorem \cite{Corlette} (see \cite[Thm. 3.7]{FH}) there exists a continuous representation $\overline{\phi}:\mathsf{Sp}(m,1)\rightarrow \mathsf{GL}_{52}(\mathbb{C})$ and a representation $\phi_0:\Delta_0 \rightarrow \mathsf{GL}_{52}(\mathbb{C})$ with compact closure such that the images $\overline{\phi}(\Delta_0)$ and $\phi_0(\Delta_0)$ commute and $\phi(\gamma)=\overline{\phi}(\gamma)\phi_0(\gamma)$ for every $\gamma \in \Delta_0$. Since $\phi'(\Delta_0)$ has noncompact closure, the representation $\overline{\phi}$ is unbounded and hence has finite kernel. In particular, the Lie algebra of $\mathsf{Sp}(m,1)$ embeds into that of $\mathsf{GL}_{52}(\mathbb{C})$. Since the dimension of the symplectic group $
\mathsf{Sp}(m)$ is $m(2m+1)$, and since $\mathbb{H}\mathbf{H}^m\cong \mathsf{Sp}(m,1)/\big(\mathsf{Sp}(m)\times \mathsf{Sp}(1)\big)$ is $4m$-dimensional, the dimension of the Lie algebra of $\mathsf{Sp}(m,1)$ is $2m^2+5m+3>2\cdot 52^2$ for $m\geq 51$. We obtain a contradiction, and hence the image of $\phi$ is finite. It follows that the image of $\rho$ is finite.\end{proof}

%However, this cannot happen since $m\geq 51$ and the dimension of the Lie algebra of $\mathsf{Sp}(m,1)$ is $2m^2+5m+3>2\cdot 52^2$. 

We are now ready to establish our main result.

\begin{proof}[Proof of Theorem \ref{mainthm1}] Let $\mathsf{G}$ be a semisimple Lie group that is isogenous to a product of rank-$1$ Lie groups and $\rho:\Gamma \ast \Delta\rightarrow \mathsf{G}$ a representation. We prove that $\rho$ cannot be discrete and faithful.

\par By our assumption on $\mathsf{G}$, one can find a finite-index subgroup $\mathsf{G}_0$ of $\mathsf{G}$, rank-$1$ Lie groups $\mathsf{G}_1,\ldots,\mathsf{G}_q$, and a continuous epimorphism $\pi:\mathsf{G}_0\rightarrow \prod_{i=1}^{q}\mathsf{G}_i$ with finite kernel. Choose finite-index subgroups $\Gamma_0$ and $\Delta_0$ of $\Gamma$ and $\Delta$, respectively, such that $\rho(\Gamma_0\ast \Delta_0) \subset \mathsf{G}_0$. It suffices to prove that the composition $\phi:=\pi\circ \rho:\Gamma_0\ast \Delta_0\rightarrow \prod_{i=1}^{q}\mathsf{G}_i$ cannot be discrete and faithful.

\par Let $\textup{pr}_i:\prod_{j=1}^{q}\mathsf{G}_j \rightarrow \mathsf{G}_i$ denote the projection onto the $i^\text{th}$ factor, let $I_1\subset \{1,\ldots,q\}$ be the (possibly empty) set of indices $i$ such that $\mathsf{G}_i$ is isogenous to $\mathsf{F}_{4}^{(-20)}$, and set $I_2:=\{1,\ldots,q\}\smallsetminus I_1$. Since $m\geq 51$, by Theorem \ref{bd1}, the representation $\textup{pr}_i\circ \phi:\Delta_0\rightarrow \mathsf{G}_i$ has finite image for every $i\in I_1$, so that the subgroup $\Delta_1:=\bigcap_{i\in I_1}\textup{ker}(\textup{pr}_i \circ \phi) < \Delta_0$ is of finite index. Moreover, by Proposition \ref{bd2}, for every $j\in I_2$, the image of the representation $\textup{pr}_j\circ \phi:\Gamma_0\rightarrow \mathsf{G}_j$ is bounded since $\mathsf{G}_j$ is not isogenous to $\mathsf{F}_4^{(-20)}$. Now choose an arbitrary infinite sequence~$(\gamma_n)_{n\in \mathbb{N}}$ of distinct elements of $\Gamma$ and a non-trivial element $\delta\in \Delta_1$. Then the terms of the sequence 
$$g_n=[\delta,\gamma_n]=\delta \gamma_n \delta^{-1}\gamma_{n}^{-1},$$ 
$n\in \mathbb{N}$, in $\Gamma_0\ast \Delta_1$ are distinct. For every $n \in \mathbb{N}$ and $i\in I_1$, we have that $\textup{pr}_i(\phi(g_n))=1$ since $\textup{pr}_i(\phi(\delta))=1$. Moreover, for every $j\in I_2$, the sequence $(\textup{pr}_i(\phi(g_n))_{n\in \mathbb{N}}$ is bounded in~$\mathsf{G}_j$ since $\textup{pr}_i(\phi(\Gamma_0))$ is bounded. It follows that $(\phi(g_n))_{n\in \mathbb{N}}$ is bounded in the product $\prod_{i=1}^{q}\mathsf{G}_i$ and hence $\pi\circ \rho$ cannot be discrete and faithful. 
\end{proof}

\begin{remark}\label{classical}
If $\Gamma_1$ and $\Gamma_2$ are infinite-covolume convex cocompact subgroups of a rank-$1$ Lie group $\mathsf{G}$, so that the $\Gamma_i$ have nonempty domain of discontinuity on the visual boundary of the symmetric space of $\mathsf{G}$, then classical arguments of Maskit \cite[Thm. VII.C.2]{Maskit} imply that the free product $\Gamma_1 * \Gamma_2$ also embeds as a convex cocompact subgroup of~$\mathsf{G}$. For any convex cocompact subgroup (in particular, any uniform lattice) $\Gamma$ of a rank-$1$ Lie group not isogenous to $\mathsf{F}_4^{(-20)}$, there is some $m \in \mathbb{N}$ such that $\Gamma$ acts convex cocompactly on $\mathbb{H}{\bf H}^m$ preserving a proper totally geodesic subspace of the latter; in particular, the free product of any two such~$\Gamma$ again admits a convex cocompact representation into a rank-$1$ Lie group. 
\end{remark}

\section{Anosov subgroups and free products}\label{freeproducts}

In this section, we justify that the free products discussed in Section \ref{proofs} admit Anosov embeddings. Using a combination theorem of Dey--Kapovich--Leeb \cite{DKL} (see also \cite{MR4647907}), we show more generally that the property of admitting an Anosov embedding into some special linear group  is preserved under taking finitely many free products. 

\begin{proposition}\label{anosovpreservedunderfreeproducts}
Let $\Gamma_1$ and $\Gamma_2$ be residually finite groups admitting $P_1$-Anosov representations $\rho_i: \Gamma_i \rightarrow \mathsf{SL}_n(\mathbb{R})$, $i=1,2$. Then $\Gamma_1 * \Gamma_2$ embeds as a $P_1$-Anosov subgroup of $\mathsf{SL}_N(\mathbb{R})$ for some $N\in \mathbb{N}$.
\end{proposition}

We remark that it follows from the theory of convex cocompactness in real projective spaces developed by Danciger--Gu\'eritaud--Kassel \cite[Thm.~1.15]{DGK}, together with a result \cite[Prop. 12.5]{DGK} announced by the same authors, that if $\Gamma_1$ and $\Gamma_2$ are $P_1$-Anosov subgroups of $\mathsf{SL}_n(\mathbb{R})$, then $\Gamma_1 * \Gamma_2$ embeds as a $P_1$-Anosov subgroup of $\mathsf{SL}_N(\mathbb{R})$, where $N = \frac{n(n+1)}{2} + 1$. More generally, Danciger--Gu\'eritaud--Kassel show that a free product $\Gamma_1\ast \Gamma_2$ of two discrete subgroups $\Gamma_1,\Gamma_2<\mathsf{SL}_N(\mathbb{R})$ that are convex cocompact in, but do not divide a properly convex domain in, $\mathbb{P}(\mathbb{R}^N)$ embeds as a discrete subgroup of $\mathsf{SL}_N(\mathbb{R})$ that is again convex cocompact in $\mathbb{P}(\mathbb{R}^N)$. Their proof will appear in \cite{DGK2}. 

We first establish Theorem \ref{anosovfreeproducts} assuming Proposition \ref{anosovpreservedunderfreeproducts}.

\begin{proof}[Proof of Theorem \ref{anosovfreeproducts}] %\normalfont{For every semisimple linear algebraic $\mathbb{R}$-group $\mathsf{G}$, one can find an integer $d=d(\mathsf{G})$ and a Lie group homomorphism $\psi:\mathsf{G}\rightarrow \mathsf{SL}_d(\mathbb{R})$ with the property that, for every Anosov subgroup $\Gamma<\mathsf{G}$, the restriction $\psi|_{\Gamma}:\Gamma\rightarrow \mathsf{SL}_d(\mathbb{R})$ is $P_1$-Anosov; see, for instance, \cite[Prop. 4.7]{GW12}.} One thus concludes from Proposition \ref{anosovpreservedunderfreeproducts} that if $\mathsf{G}_1$ and $\mathsf{G}_2$ are semisimple linear algebraic $\mathbb{R}$-groups and $\Gamma_i<\mathsf{G}_i$, $i=1,2$, is a ${\Theta_i}$-Anosov subgroup of $\mathsf{G}_i$, then $\Gamma_1\ast \Gamma_2$ embeds as a $P_1$-Anosov subgroup of $\mathsf{SL}_d(\mathbb{R})$ for some $d\in \mathbb{N}$.
Let $\mathsf{G}_i$ be a semisimple linear algebraic $\mathbb{R}$-group and $\Gamma_i < \mathsf{G}_i$ a ${\Theta_i}$-Anosov subgroup of $\mathsf{G}_i$ for $i=1,2$ and some non-empty set ${\Theta_i}$ of simple restricted roots of~$\mathfrak{g}_i$. For $i=1,2$, one can find an integer $n_i$ and a Lie group homomorphism $\psi_i: \mathsf{G}_i\rightarrow \mathsf{SL}_{n_i}(\mathbb{R})$ with the property that the restriction $\psi_i|_{\Gamma_i}:\Gamma_i\rightarrow \mathsf{SL}_{n_i}(\mathbb{R})$ is $P_1$-Anosov; see, for instance, \cite[Prop. 4.7]{GW12}. By adding 1's on the diagonal, we can assume $n_1 = n_2 =: n$, and the resulting representations $\rho_i: \Gamma_i \rightarrow \mathsf{SL}_n(\mathbb{R})$ remain $P_1$-Anosov; see \cite[Prop. 4.4]{GW12}. Moreover, since $\Gamma_i < \mathsf{G}_i$ and the $\mathsf{G}_i$ are assumed to be linear, we have that the $\Gamma_i$ are linear and thus residually finite \cite{malcev1940isomorphic}. It then follows from Proposition \ref{anosovpreservedunderfreeproducts} that $\Gamma_1 * \Gamma_2$ embeds as a $P_1$-Anosov subgroup of $\mathsf{SL}_N(\mathbb{R})$ for some $N \in \mathbb{N}$.
\end{proof}

\begin{proof}[Proof of Proposition \ref{anosovpreservedunderfreeproducts}]
We assume throughout that the $\Gamma_i$ are infinite. Indeed, if the $\Gamma_i$ are both finite, so that they both embed discretely in $\mathrm{O}(M)$ for some $M \in \mathbb{N}$, then ${\Gamma_1*\Gamma_2}$ embeds as a convex cocompact subgroup of $\mathsf{O}(M, 1)$ (see Remark \ref{classical}), and hence as a $P_1$-Anosov subgroup of $\mathsf{SL}_{M+2}(\mathbb{R})$. Moreover, if $\Gamma_1$ is infinite and $\Gamma_2$ is finite, then, by the Kurosh subgroup theorem (see the discussion immediately preceding Theorem 3.1 in \cite{sw79}), the kernel of the projection $\Gamma_1 * \Gamma_2 \rightarrow \Gamma_2$ is of the form
\begin{equation*}
\Gamma_1 * \cdots * \Gamma_1 * \mathbb{Z} * \cdots * \mathbb{Z},
\end{equation*}
so we have reduced to the case where the factors are all infinite, as the property of admitting a $P_1$-Anosov embedding into some special linear group passes to finite-index supergroups; see Lemma \ref{comm}.

Since the $\Gamma_i$ are residually finite, we can find finite-index normal subgroups $\Delta_i \triangleleft \Gamma_i$ intersecting the (finite) kernel of $\rho_i$ trivially. We identify the $\Delta_i$ with their images under the $\rho_i$. Now replace the $\Delta_i$ with their images under the symmetric square $\mathsf{SL}_n(\mathbb{R}) \rightarrow\mathsf{SL}(V) \cong \mathsf{SL}_d(\mathbb{R})$, where $V$ is the space of symmetric $(n \times n)$ real matrices and $d = \frac{n(n+1)}{2}$. Since this representation of $\mathsf{SL}_n(\mathbb{R})$ preserves the positive-definite cone $\mathcal{C} \subset V$, we have that the $\Delta_i$ preserve $\Omega := \mathbb{P}(\mathcal{C}) \subset \mathbb{P}(\mathbb{R}^d)$; the latter is a nonempty properly convex domain in $ \mathbb{P}(\mathbb{R}^d)$. Moreover, the $\Delta_i$ remain $P_1$-Anosov in $\mathsf{SL}_d(\mathbb{R})$ (again, by \cite[Prop. 4.4]{GW12}). Now identify $\mathbb{R}^d$ with the linear hyperplane $\Pi := \{x_1 = 0\}\subset \mathbb{R}^{d+1}$ via the map $(x_2, \ldots, x_{d+1}) \mapsto (0, x_2, \ldots, x_{d+1})$, and view $\mathsf{SL}_d(\mathbb{R})$, and hence the $\Gamma_i$, as being included in $\mathsf{SL}_{d+1}(\mathbb{R})$ via the map
\begin{equation*}
g \mapsto \begin{pmatrix} 1 & 0 \\ 0 & g \end{pmatrix}.
\end{equation*}
Then the $\Delta_i$ are $P_1$-Anosov, and hence $P_d$-Anosov (see \cite[Lem.~3.18~(i)]{GW12}), in $\mathsf{SL}_{d+1}(\mathbb{R})$. 

Let $\mathcal{F}$ be the flag manifold of $\mathsf{SL}_{d+1}(\mathbb{R})$ consisting of all pairs $(\ell', \pi)$ where $\ell' \in \mathbb{P}(\mathbb{R}^{d+1})$ and $\pi$ is a projective hyperplane in $\mathbb{P}(\mathbb{R}^{d+1})$ containing $\ell'$, and let $\Lambda_i \subset \mathcal{F}$ be the limit set of~$\Delta_i$ in $\mathcal{F}$. For each pair $(\ell', \pi)\in \Lambda_i$, we have that $\ell' \in \partial \Omega$, that $\pi \cap \Omega = \emptyset$, and that $\pi$ contains the point $[1: 0 : \ldots : 0] \in \mathbb{P}(\mathbb{R}^{d+1})$. Choose a point $\ell \in \Omega$ and points $\ell^\pm$ on the projective line $L \subset \mathbb{P}(\mathbb{R}^{d+1})$ joining $\ell$ to the point ${[1: 0 : \ldots : 0 ]}$ so that all four of the points just mentioned are distinct. Choose also projective hyperplanes $\pi^\pm \subset \mathbb{P}(\mathbb{R}^{d+1})$ containing $\ell^\pm$ and whose intersection with~$\mathbb{P}(\Pi)$ is disjoint from $\overline{\Omega}$. 

Under the above assumptions, the flags $(\ell^\pm, \pi^\pm) \in \mathcal{F}$ are transverse, and we can find an element $h \in \mathsf{SL}_{d+1}(\mathbb{R})$ that is simultaneously $P_1$- and $P_d$-proximal whose attracting and repelling fixed points in $\mathcal{F}$ are $(\ell^\pm, \pi^\pm)$. Moreover, the sets $\{ (\ell^\pm, \pi^\pm) \}$ and $\Lambda_1 \cup \Lambda_2$ are antipodal in $\mathcal{F}$, in the sense that each flag in one set is transverse to each flag in the other. It follows from \cite[Lem.~4.24]{DKL} that one can then find a neighborhood $U \subset \mathcal{F}$ of $\{(\ell^\pm, \pi^\pm)\}$ so that the sets $U$ and $\Lambda_1 \cup \Lambda_2$ remain antipodal in $\mathcal{F}$.

Up to replacing $h$ with one of its powers, we have that $h \Lambda_2 \subset U$. Since $h \Lambda_2$ is the limit set of $h \Delta_2 h^{-1}$ in $\mathcal{F}$, it follows from \cite[Thm.~1.3]{DKL} that there are finite-index characteristic subgroups $\Gamma_i'$ of $\Delta_i$ so that $\langle \Gamma_i', h \Gamma_2 ' h^{-1} \rangle < \mathsf{SL}_{d+1}(\mathbb{R})$ is naturally isomorphic to $\Gamma_1' * \Gamma_2'$ and is $P_1$-Anosov in $\mathsf{SL}_{d+1}(\mathbb{R})$. 

Let $\Gamma_0 < \Gamma_1 * \Gamma_2$ be the intersection of the kernels of the compositions $\Gamma_1 * \Gamma_2 \rightarrow \Gamma_i \rightarrow \Gamma_i / \Gamma_i'$. Then $\Gamma_0$ is of finite index in $\Gamma_1*\Gamma_2$ and, again by the Kurosh subgroup theorem, is isomorphic to a group of the form
\begin{equation}\label{formoffreeproduct}
\Gamma_1 ' * \cdots * \Gamma_1' * \Gamma_2' * \cdots * \Gamma_2' * \mathbb{Z} * \cdots * \mathbb{Z}.
\end{equation}
Since the $\Gamma_i'$ are both infinite, any group of the above form embeds as a quasiconvex subgroup of the Gromov-hyperbolic group $\Gamma_1' * \Gamma_2'$; indeed, for any $\gamma_i \in \Gamma_i'$, $i=1,2$, of infinite order, the subgroup
\begin{equation*}
\left\langle \gamma_2 \Gamma_1' \gamma_2^{-1}, \ldots, \gamma_2^r \Gamma_1 ' \gamma_2^{-r}, \gamma_1 \Gamma_2' \gamma_1, \ldots, \gamma_1^s \Gamma_2' \gamma_1^{-s}, \gamma_1^{s+1} \gamma_2 \gamma_1^{-(s+1)}, \ldots, \gamma_1^{s+q} \gamma_2 \gamma_1^{-(s+q)} \right\rangle
\end{equation*}
of $\Gamma_1' * \Gamma_2'$ is quasiconvex and is naturally isomorphic to a free product of the form (\ref{formoffreeproduct}) (that we may find $\gamma_i$ of infinite order in $\Gamma_i$ follows from the fact that the $\Gamma_i$ are infinite finitely generated linear groups, for instance). Since we have already found a $P_1$-Anosov embedding of the latter into a special linear group, we conclude that $\Gamma_0$ also admits such a representation, and hence so does the finite-index supergroup $\Gamma_1 * \Gamma_2$ by Lemma \ref{comm}. 
\end{proof}

\bibliography{indiscretebiblio}{}
\bibliographystyle{siam}

\end{document}